\documentclass{amsart}
\usepackage{amssymb, amsmath}
\usepackage[dvips]{graphicx}
\usepackage{amsfonts}
\usepackage{latexsym}
\usepackage{color}
\newtheorem{theorem}{Theorem}	
\newtheorem{lemma}{Lemma}[section]		
\newtheorem{corollary}{Corollary}		
\newtheorem{proposition}{Proposition}

\newtheorem*{example}{Example}
\newtheorem*{question}{Question}

\begin{document}
\title[From separation center set to the upper and lower bound theorems]
{
From spherical separation center set to the upper and lower bound theorems}
\author{
Huhe Han}
\address{College of Science, Northwest Agriculture and Forestry University, China}
\email{han-huhe@nwafu.edu.cn}
\begin{abstract} 
For any disjoint spherical closed convex set $C_1, C_2$, set
$\mathcal{C}(C_1, C_2)= 
\{
P\in S^{d}| P\cdot Q_1\geq 0\ \mbox{for\ any}\ Q_1\in C_ 1\mbox{and}\ P\cdot Q_2\leq 0\ \mbox{for\ any}\ Q\in C_2
\}$. 
Geometrically, the union $\mathcal{C}(C_1, C_2)\cup \mathcal{C}(C_2, C_1)$ is the set consisting of the center of the hemisphere which separates $C_1, C_2$. 
We call  $\mathcal{C}(C_1, C_2)\cup \mathcal{C}(C_2, C_1)$ 
the separation center set of $C_1, C_2$.
It is known that $\mathcal{C}(C_1, C_2)$ is a spherical convex set. 
In this paper, we first prove that for any spherical closed convex set $C$ there exist disjoint 
spherical closed convex sets $C_1, C_2$ such that $C=\mathcal{C}(C_1, C_2)$. 
We also study the properties of the separation center set of (1) the two disjoint spherical convex set and (2) the finite closed set. The motivation for considering the case (1) is that it is the classical case, and the vertices of a spherical polytope may be considered as the case (2).
A disjoint spherical polytopes pair $(\mathcal{P}_1, \mathcal{P}_2)$ is said to be a
 spherical face-partition pair of $\mathcal{P}$ if $\mathcal{P}$ is the spherical convex hull of the union 
$\mathcal{P}_1\cup -\mathcal{P}_2$ and 
$\mathcal{P}_1$ or -$\mathcal{P}_2$ is a face of $\mathcal{P}$.
The number of spherical face-partition pair $(\mathcal{P}_1, \mathcal{P}_2)$ of $\mathcal{P}$
is denoted by $\#_f\mathcal{P}$. 
For any simple spherical polytope $\mathcal{P}$ with $p$ facets, 
applying the Upper Bound Theorem and the Lower Bound Theorem, 
we provide the upper bound and the lower bound of $\#_f\mathcal{P}$.
\end{abstract}
\maketitle  
\section{Introduction}
Throughout this paper, all convex sets are assumed closed without otherwise stated.
\par
A polytope $\mathcal{P}$ in $\mathbb{R}^d$ is said to be {\it simplicial} if each proper face of $\mathcal{P}$ is a simplex. 
A polytope $\mathcal{P}$ in $\mathbb{R}^d$ is said to be {\it simple} if each vertex of $\mathcal{P}$ is contained exactly in $d$ facets. 
It is well-known that if $\mathcal{P}_1$ and $\mathcal{P}_2$ are duals polytopes, then 
$\mathcal{P}_1$ is simple if and only if $\mathcal{P}_2$ is simplicial. 
A $d-$polytope $\mathcal{P}$ is said to be {\it $k-$ neighbourly} 
if convex hull of $X$ is a proper face of $\mathcal{P}$ for any $k-$subset $X$ of vertices of $\mathcal{P}$; and 
$\left \lfloor \frac{d}{2} \right \rfloor$-neighbourly $d-$polytope is called {\it neighbourly polytope} (\cite{arne}).
Following \cite{arne}, we define for $j\geq0$
\[
\Phi_j(d,p):=\sum_{i=0}^{n}\binom{i}{j}\binom{p-d+i-1}{j} +\sum_{i=0}^{m}\binom{d-i}{j}\binom{p-d+i-1}{j}.
\]
Here 
$m:=\left \lfloor \frac{d-1}{2} \right \rfloor,n:=\left \lfloor \frac{d}{2} \right \rfloor$;
\[
\varphi_j(d,p):=
\binom{d}{j+1} p-\binom{d+1}{j+1} (d-1-j),\ \ \ \ \ \ \ j=1,\dots,d-2.
\]
Then the Upper Bound Theorem and the Lower Bound Theorem may be stated as follows (\cite{arne}):
\begin{theorem}[The Upper Bound Theorem]\label{UBT}
For any simplicial $d-$polytope $\mathcal{P}$ with $p$ vertices, the following inequality holds:
\[
f_j(\mathcal{P})\leq \Phi_{d-1-j}(d, p),\ \ \ j=1,\dots,d-1. 
\]
If $\mathcal{P}$ is dual neighbourly , then 
\[
f_j(\mathcal{P})= \Phi_{d-1-j}(d, p),\ \ \ j=1,\dots,d-1.
\]
If $\mathcal{P}$ is not dual neighbourly , then 
\[
f_j(\mathcal{P})<\Phi_{d-1-j}(d, p),\ \ \ j=1,\dots,m+1.
\]
(and possibly also for larger values of $j$).
\end{theorem}
\begin{theorem}[The Lower Bound Theorem]\label{LBT}
For any simplicial $d-$polytope $\mathcal{P}$ with $p$ vertices, we have
\[
f_j(\mathcal{P})\geq \varphi_{d-1-j}(d, p),\ \ \ j=1,\dots,d-1. 
\]
Moreover, there are simple d-polytopes $P$ with $p$ vertices such that 
\[
f_j(\mathcal{P})=\varphi_{d-1-j}(d, p),\ \ \ j=1,\dots, d-1. 
\]
\end{theorem}
The Upper Bound Theorem and the Lower Bound Theorem are main achievements in the modern theory of convex polytopes, they were proved by McMullen (\cite{mcmullen}) and Barnette (\cite{barnette1, barnette2}) respectively.  
More details on the Upper Bound Theorem, the Lower Bound Theorem and their related topics see for instance \cite{bron,ewald, matousek, novik}.
\par
Let $S^d$ be the unit sphere of $\mathbb{R}^{d+1}$ and
let $C_1, C_2$ of $S^{d}$ be two disjoint spherical closed convex sets. It is known that the set 
\[
\mathcal{C}(C_1, C_2)= 
\{
P\in S^{d}| P\cdot Q_1\geq 0\ \mbox{for\ any}\ Q\in C_ 1\mbox{and}\ P\cdot Q_2\leq 0\ \mbox{for\ any}\ Q\in C_2
\}
\]
is a spherical convex set(\cite{hannishimura2020, zalinescu}), 
where $P\cdot Q$ stands for the standard scalar product 
for two vectors $P, Q\in \mathbb{R}^{d+1}$.  
Geometrically, $\mathcal{C}(C_1, C_2)\cup -\mathcal{C}(C_1, C_2)$
is the set consisting of 
the center of the hemisphere which separates $C_1$ and $C_2$. 
It is natural to ask: for any given spherical closed convex set $C$, are there exist disjoint spherical convex sets 
$C_1, C_2$ such that $C=\mathcal{C}(C_1, C_2)$? 
In this paper, we give an affirmative answer to this question (Proposition \ref{recog}), and provid the largest disjoint convex sets pair separated by $C$ (Corollary \ref{corolargest}). 
Especially, if $C$ is a spherical polytope then $C= \mathcal{C}(\mathcal{P}_1, \mathcal{P}_2)$ if and only if 
$C$ is the spherical convex hull of $\mathcal{P}_1\cup -\mathcal{P}_2$, where $\mathcal{P}_1, \mathcal{P}_2$ are disjoint spherical polytopes (Theorem \ref{polytopesepa}).
A disjoint spherical polytopes pair $(\mathcal{P}_1, \mathcal{P}_2)$ is said to be a
{\it spherical face-partition pair} of $\mathcal{P}$ if $\mathcal{P}$ is the spherical convex hull of the union 
$\mathcal{P}_1\cup -\mathcal{P}_2$ and 
$\mathcal{P}_1$ or -$\mathcal{P}_2$ is a face of $\mathcal{P}$.
The {\it sum of the spherical face-partition pair } $(\mathcal{P}_1, \mathcal{P}_2)$ of $\mathcal{P}$
is denoted by $\#_f\mathcal{P}$.
Then combine the central projection and the Upper and Lower Bound Theorems, we have the following:
\begin{theorem}\label{theoremulb}
Let $\mathcal{P}$ be a simple spherical polytope with $p$ facets. 
Then we have the following inequality,
\[
2\left(p+\sum_{j=1}^{d-1}\varphi_{d-1-j}(d,p)\right)\leq\#_f\mathcal{P}\leq 2 
\left(p+\sum_{j=1}^{d-1}\Phi_{d-1-j}(d,p)\right).
\]
\end{theorem}
\bigskip
This paper organized as follows. 
In Section 2, the preliminaries are given.
In Section 3, there are two cases of the separation center sets are considered.
Namely, (i) the separation center set of two disjoint spherical convex sets and (ii) the separation center set of finite sets. 
The case (i) is the classic case, and the vertices of a spherical polytope may be considered as the case (ii). 
In Section 4, some results between spherical polytopes and the separation center set of spherical polytopes are given.
The proof of Theorem \ref{theoremulb} is given in Section 5. 

\section{preliminaries}
\subsection{Spherical polytopes}
Let $H(P)$ be the set consisting of $Q\in S^d$ satisfying $P\cdot Q\geq 0$.   
The central projection 
$\alpha_{{}_P}: S^d\backslash H(-P)\to P+T_{{}_{P}}(S^d)$, defined by
\[
\alpha_{{}_{P}}(Q)=
\frac{1}{P\cdot Q}Q
\]
for any $Q$ such that $P\cdot Q >0$,
where $
T_{{}_{P}}(S^d)
=
\{Q\in \mathbb{R}^{n+1}\mid Q\cdot P=0
\}
$
is the tangent vector space of $S^d$ at $P$ 
and 
$P+T_{{}_{P}}(S^d)$ is the tangent affine space of $S^d$ at $P$ defined by 
$\left\{P+x\; \left|\; x\in T_{{}_{P}}(S^d)\right.\right\}$. 
A spherical polytope $\widetilde{\mathcal{P}}$ is said to be 
{\it simplicial} if 
the polytope $\alpha_M(\widetilde{\mathcal{P}})$ is simplicial in $M+T_{{}_{M}}(S^d)$;
a spherical polytope $\widetilde{\mathcal{P}}$ is said to be 
{\it simple} if 
the polytope $\alpha_M(\widetilde{\mathcal{P}})$ is simple in $M+T_{{}_{M}}(S^d)$,
where $M$ is an (relative) interior point of $\widetilde{\mathcal{P}}$.  
\par 
A subset $W$ of $S^d$ is said to be 
\textit{hemispherical} if there exists a point $P\in S^d$ 
such that 
\[
W\cap H(P)=\emptyset .
\]
A hemispherical subset $W$ of $S^d$ is said to be \textit{spherical convex} 
if for any 
$P, Q\in W$ and any $t\in [0,1]$ 
the unit vector 
\[
\frac{tP+(1-t)Q}{||tP+(1-t)Q||}
\]
is contained in $W$.  
Here, $|| \cdot ||$ denotes the standard $(d+1)-$dimensional Euclidean norm. 
Since $W$ is hemispherical, the denominator $||tP+(1-t)Q||$ always positive.
Let $\{W_i\}_{i=1}^m$ be a finite closed sets of $S^d$ such that $W_i\cap W_j=\emptyset$ 
for any $i, j\in \{1, \dots, m\}$,
and let $(I, J)$ be a {\it partition} of $\{1,\dots, m\}$, namely, 
$I$ and $J$ are nonempty subsets of $\{1,\dots, m\}$ such that $I\cap J =\emptyset $ 
and $I\cup J=\{1,\dots, m\}$. 
A point $P$ of $S^d$ is said to be a 
${\it separation \ center} $ of $\{W_i\}_{i=1}^m$ if there exists a partition $(I, J)$ of $\{1,\dots, m\}$ such that 
\[ 
P\cdot Q\geq 0\ \mbox{for\ any}\ Q\in W_i, i\in I\ \mbox{and}\ P\cdot R\leq 0\ \mbox{for\ any}
\ R\in W_j, j\in J.
\]
Then it is clear that if $P$ is a separation center of $\{W_i\}_{i=1}^m$ then is so for $-P$.
The set consisting of separation centers of 
$\{W_i\}_{i=1}^m$ is called {\it the separation center set} of $\{W_i\}_{i=1}^m$.
By definition it is clear that separation center set is a spherical central symmetry set. 
In particular,
the subset 
\[
\{P\in S^d\mid P\cdot Q\geq 0\ \mbox{for\ any}\ Q\in W_1\ \mbox{and}\ P\cdot R\leq 0\ \mbox{for\ any}
\ R\in W_2\},
\]
of the separation center set of two spherical closed sets 
$W_1$ and $W_2$ 
is called {\it semi-separation center set} of $W_1, W_2$, 
denoted by $\mathcal{C}(W_1, W_2)$.
Therefore the union 
\[
\mathcal{C}(W_1, W_2)\cup\mathcal{C}(W_2, W_1)
\]
 is the separation center set of 
$W_1, W_2$. 
%
 \subsection{Spherical polar sets}
 For any nonempty hemispherical set $W$, the 
{\it spherical polar set} of $W$ is the following spherical convex set, 
denoted by $W^{\circ}$,
\[
\bigcap_{P\in W}H(P).
\] 
It is clear that $W^\circ$ is always nonempty convex set. 
\begin{lemma}\label{lemmaobvi}
The following equality holds for any $P\in S^d$ and $0<\varepsilon<
\pi/2$,
\[
B(H(P),\varepsilon)=\bigcup_{P_1\in B(P, \varepsilon)}H(P_1).
\]
\end{lemma}
\begin{proof}
Let $Q\in B(H(P), \varepsilon)$. 
If $Q\in H(P)$, it is clear that $Q\in \cup_{P_1 \in B(P, \varepsilon)}H(P_1)$. 
If $Q\notin H(P)$, set $r=\arccos(P\cdot Q)$, then $\pi/2<r<\pi/2+\varepsilon$.
Set $R=PQ\cap \partial\overline{B(P, r-\pi/2)}$. 
Then it follows that 
\[
\arccos(Q\cdot R)=\arccos(P\cdot Q)-\arccos(R\cdot P)=r-(r-\frac{\pi}{2})=\frac{\pi}{2}.
\]
This implies $Q\in H(R)$. 
Since $R\in \partial\overline{B(P, r-\pi/2)}\subset B(P, \varepsilon)$, 
we have that $Q\in \cup_{P_1\in B(P, \varepsilon)}H(P_1).$ 
Conversely, let $Q\in \cup_{P_1\in B(P, \varepsilon)}H(P_1)$. 
Then there exists a point $R\in B(P, \varepsilon)$ such that $Q\in H(R)$. 
This means that 
\[
\arccos(P\cdot Q)\leq \arccos(P\cdot R)+\arccos(R\cdot Q)\leq \varepsilon+\frac{\pi}{2}.
\]
Therefore, $\cup_{P_1\in B(P, \varepsilon)}H(P_1)$ is a subset of $B(H(P),\varepsilon)$.
\end{proof}
The following properties of spherical polar sets are needed in later sections.
\begin{lemma}[\cite{nishimurasakemi2}]\label{lemmasconv}
Let $W$ be nonempty hemispherical subsets of $S^{d}$.        
Then, the equality $W^{\circ}= \left(\mbox{\rm s-conv}(W)\right)^{\circ}$ holds, 
where $\mbox{\rm s-conv}(W)$ is the spherical convex hull of $W$, namely,
\[
\mbox{\rm s-conv}(W)=
\left\{
\frac{\sum_{i=1}^m \lambda_iP_i}{||\sum_{i=1}^m \lambda_iP_i||}\in S^d\mid \sum_{i=1}^m \lambda_i=1, P_i\in W
\right\}.
\]

\end{lemma}
\begin{lemma}\label{lemma2.1}
Let $W$ be a nonempty hemispherical set of $S^d$.
Then $(-W)^\circ=-W^\circ$.
\end{lemma}
\begin{proof}
Let $P$ be a point of $(-W)^{\circ}$. 
Then the following holds.
  \begin{eqnarray*}
P\in  (-W)^{\circ}& \Longleftrightarrow &
   P\cdot Q\geq 0\ for\ any\ Q\ \in (-W)\\
           & \Longleftrightarrow & P\cdot (-Q)\geq 0\ for\ any\ Q\in W\\
           & \Longleftrightarrow & -P\cdot Q\geq 0\ for\ any\ Q\in W\\
           & \Longleftrightarrow & -P\in W^\circ\\
           & \Longleftrightarrow & P\in -W^\circ.
\end{eqnarray*}
\end{proof}
\par
\noindent
By the proof of Lemma 4.1 of \cite{han1}, we have the following: 
\begin{lemma}\label{lemma1}
Let $W_1$ and $W_2$ be nonempty subsets of $S^d$. Then 
\[
W_1^{\circ}\cap W_2^{\circ}=\left(W_1\cup W_2\right)^{\circ}.
\]
\end{lemma}
In fact Lemma \ref{lemma1} follows from the following:
\begin{eqnarray*}
  P\in  W_1^{\circ}\cap W_2^{\circ}& \Longleftrightarrow &
   P\cdot Q\geq 0\ for\ any\ Q\ \in W_1\ and\ P\cdot R\geq 0\ for\ any\ R \in W_2\\
           & \Longleftrightarrow & P\cdot Q\geq 0\ for\ any\ Q\in \left(W_1\cup W_2\right)\\
           & \Longleftrightarrow & P\in \left(W_1\cup W_2\right)^{\circ}.
\end{eqnarray*}
\begin{lemma}[\cite{nishimurasakemi2}]\label{lemmainclusion}
Let $W_{1}, W_{2}$ be nonempty subsets of $S^{d}$.    
Suppose that the inclusion ${W_{1}\subset W_{2}}$ holds.    
Then, the inclusion $W_{2}^{\circ}\subset W_{1}^{\circ}$ holds.
\end{lemma}

\begin{lemma}[\cite{nishimurasakemi2}]\label{lemmas-conv}
For any nonempty closed hemispherical subset $W \subset S^{d}$, 
the equality $\mbox{ \rm s-conv}(W)= \left(
\mbox{ \rm s-conv}\left( W \right) \right)^{\circ\circ }$ holds.
\end{lemma}
It is known that the spherical dual transform is an isometry:
\begin{theorem}[\cite{hannishimura}] \label{isometry}
Let $\mathcal{H}(S^d)$ be the set of non-empty closed subsets of $S^d$ 
and let
\begin{align*}
\mathcal{H}(S^d, P) = \{
W\in \mathcal{H}(S^d) \mid 
 W\ \mbox{is a sphrical}\ &\mbox{convex body and}\\ 
& W \cap H(-P)=\emptyset,
 P\in int(W)
\}.
\end{align*}
Then the mapping 
\[
\bigcirc_P:\overline{\mathcal{H}(S^d, P)}\to \overline{\mathcal{H}(S^d,P)}
\]
is an isometry with respect to Pompeiu-Hausdorff metric, where
and 
$\overline{\mathcal{H}(S^d, P)}$ is the closure of 
$\mathcal{H}(S^d, P)$ in $\mathcal{H}(S^d)$.
\end{theorem}

\section{Separation center sets}
In this section, the properties of the two cases of the separation center sets are studied. 
The first one is the case $m=2$ and $W_1, W_2$ are spherical closed convex sets. 
This is an essential situation for separability of spherical convex sets. 
The another one is the case that spherical finite sets, $W_i=\{P_i\}  (i\in \{1, \dots, m\}$. 
The motivation for considering this case is that it is naturally related to the face number of spherical  polytopes. 
Moreover, 
if $W_i=\{P_i\}  (i\in \{1, \dots, m\}$ is a hemispherical finite set, 
then combine with central projection, it is naturally related to the vertices of Euclidean polytopes,
which is an important branch of Euclidean convex geometry.
\subsection{Case 1: $m=2$ and $W_1, W_2$ are spherical convex sets}
In \cite{hannishimura2020, zalinescu}, the separability of two 
nonempty closed (open)  and spherical convex subsets are studied. 
\begin{proposition}[\cite{hannishimura2020, zalinescu}]\label{sstheorem}
Let $W_1, W_2\subset S^d$ be 
two non-empty   
closed (resp., open)  and spherical convex subsets.    
Then, the following (1) and (2) are equivalent.   
\begin{enumerate}
\item[(1)] 
$W_1\cap W_2=\emptyset$.   
\item[(2)] The subset consisting of points 
$P\in S^d$ such that $P\cdot Q>0$ for any $Q\in W_1$ and $P\cdot R<0$ for any 
$R\in W_2$ is non-empty open (resp., closed) and spherical convex. 
\end{enumerate}   
\end{proposition}
As a corollary, the following holds: 
\begin{lemma}[\cite{lassak2015}]\label{lassaklemma}
Every two convex bodies on the sphere $S^d$ with empty intersection of their interiors are subsets of some two opposite hemispheres.
\end{lemma}
\begin{proposition}\label{theorem1}
Let $W_1, W_2$ be two nonempty closed 
and spherical convex sets of $S^d$. Then the following two assertions are equivalent.
\begin{enumerate}
\item $\mathcal{C}(W_1, W_2)$ is a nonempty set.
\item The intersections $W_1\cap \mbox{int}(W_2)$ and 
$\mbox{int}(W_1)\cap W_2$ are empty set, 
where $\mbox{int}(W)$ means the interior set of $W$.
\end{enumerate}
\end{proposition}
\begin{proof}
First, we prove that (1) implies (2). 
By the assumption, there exists a point $P\in S^d$ such that 
\[
P\cdot Q\geq 0\ \mbox{for\ any}\ Q\in W_1\ \mbox{and}\ P\cdot R\leq 0\ \mbox{for\ any}
\ R\in W_2.
\]
This implies that the intersection $W_1\cap W_2$ is a subset of the boundary of $H(P)$, 
\[
\partial H(P)=\{M\in S^d\mid M\cdot P=0\}.
\]
Therefore, the intersection $W_1\cap W_2$ does not contain interior point. 
In particular, the intersections $W_1\cap \mbox{int}(W_2)$ and 
$\mbox{int}(W_1)\cap W_2$ are empty.
\indent
\par
Next, we prove that (2) implies (1). 
\par
Case 1:  $W_1$ (or $W_2$) is a closed set in $S^d$ with empty interior. By the assumption, 
there exists a point $P\in S^d$ such that $W_1\subset \partial H(P)$ and $W_2 \subset H(P)$, 
 namely,
\[
P\cdot Q= 0\ \mbox{for\ any}\ Q\in W_1\ \mbox{and}\ P\cdot R\geq 0\ \mbox{for\ any}
\ R\in W_2,
\]
where $\partial H(P)$ means the boundary of $H(P)$.
Therefore, it follows that $P\in \mathcal{C}(W_1, W_2)$, 
the semi-separation center set of $W_1, W_2$ is nonempty.
\par
Case 2: $W_1, W_2$ include interior points.
Since $W_1\cap W_2$ has no interior points,
by Lemma \ref{lassaklemma}, 
there exists a point
$P\in S^d$ such that 
$W_1 \subset H(P)$ and $ W_2\subset H(-P)$, 
that is,
\[
P\cdot Q\geq 0\ \mbox{for\ any}\ Q\in W_1\ \mbox{and}\ P\cdot R\leq 0\ \mbox{for\ any}
\ R\in W_2.
\]
Therefore, separation center set $\mathcal{C}(W_1, W_2)$ is nonempty.
\end{proof}
The following result implies that for any nonempty spherical closed convex set $\mathcal{X}$, 
 the union $\widehat{\mathcal{X}}=\mathcal{X}\cup -\mathcal{X}$ can be regarded as
the separation center set of suitable two nonempty spherical convex sets.
\begin{proposition}\label{recog}
Let $\mathcal{X}=W_1\cap W_2$ be a nonempty set, 
where $W_1, W_2$ are nonempty closed  
and spherical convex subsets of $S^d$. 
Then $\widehat{\mathcal{X}}$ is the separation center set of $W_1^\circ$
and $-W_2^\circ$.
\end{proposition}
\begin{proof}
Notice that $W=W^{\circ\circ}$, 
for any nonempty closed and spherical 
convex set $W$ of $S^d$ (Lemma \ref{lemmas-conv}).
Since
$W_1, W_2$ are 
two nonempty closed  
spherical convex sets, 
by Lemma \ref{lemma2.1},
the following equality holds,
\[
W_i=\left(-(-W_i)\right)^{\circ\circ}=\left((-(-W_i))^\circ\right)^{\circ}=(-(-W_i^\circ))^\circ , i\in \{1,2\}.
\] 
Then it follows that, 
\[
\mathcal{X}=W_1\cap W_2=(W_1^\circ)^\circ \cap
(-(-W_2^\circ))^\circ.
\]
Since
\begin{eqnarray*}
  P\in  \mathcal{X}& \Longleftrightarrow &
   P\cdot Q\geq 0\ for\ any\ Q\ \in W_1^\circ\ and\ P\cdot R\geq 0\ for\ any\ R \in (-(-W_2^\circ))\\
           & \Longleftrightarrow & P\cdot Q\geq 0\ for\ any\ Q\ \in W_1^\circ\ and\ P\cdot R\leq 0\ for\ any\ R \in (-W_2^\circ),
\end{eqnarray*}
it follows that any point $P$ is an element of $\mathcal{X}$ if and only if $P$ is an element of 
the semi-separation center set 
$\mathcal{C}(W_1^\circ, -W_2^\circ)$ .
Therefore, $\widehat{\mathcal{X}}$ 
is the separation center set of $W_1^\circ$
and $-W_2^\circ$.
\end{proof}
\begin{proposition}\label{propo1}
Let $W_1, W_2$ be two nonempty closed 
and spherical convex subsets of $S^d$. Then the following hold:
\begin{enumerate} 
\item $\mathcal{C}(W_1, W_2)=-\mathcal{C}(W_2, W_1)=-\mathcal{C}(-W_1, -W_2),$
\item
$\mathcal{C}(W_1, W_2)\subset W_1^\circ$ and 
$ (-\mathcal{C}(W_1, W_2))\subset W_2^\circ$,
\item  
$W_1\subset \mathcal{C}(W_1, W_2)^\circ$ and 
$W_2\subset (-\mathcal{C}(W_1, W_2)^\circ)$, 
\item $\mathcal{C}(W_1, W_2)= W_1^\circ$ iff $-W_2\subset W_1$,
\item $\mathcal{C}(W_1, W_2)^\circ
=\mbox{ \rm s-conv}\left(W_1\cup (-W_2)\right)$,
\item $\mathcal{C}(W_1, W_2)
=(\mbox{ \rm s-conv}\left(W_1\cup (-W_2)\right))^\circ$.
\end{enumerate}
\end{proposition}
\begin{proof}
(1) Easily follows from definition of $\mathcal{C}(W_1, W_2)$. 
\par
(2) By the definition of $\mathcal{C}(W_1, W_2)$, it follows that 
\[
\mathcal{C}(W_1, W_2)\subset W_1^\circ\ \mbox{and}\ \mathcal{C}(W_1, W_2)\subset (-W_2^\circ).
\] 
\par
(3) By the assertion (2), Lemma \ref{lemma2.1}, Lemma \ref{lemmainclusion} and  Lemma \ref{lemmas-conv}, 
we have the inclusion
\[
W_1\subset \mathcal{C}(W_1, W_2)^\circ \ \mbox{and}\ 
W_2\subset (-\mathcal{C}(W_1, W_2)^\circ).
\]
\par
(4) Since $\mathcal{C}(W_1, W_2)=W_1^{\circ}\cap (-W_2)^{\circ}$, 
it follows that $\mathcal{C}(W_1, W_2)=W_1^{\circ}$ 
if and only if $W_1^{\circ}\subset (-W_2)^{\circ}$.
By Lemma \ref{lemmainclusion}, the assertion (4) follows. 
\par
(5) From Proposition \ref{recog}, Lemma \ref{lemma1}, Lemma \ref{lemmasconv} and Lemma \ref{lemmas-conv}, 
we deduce that 
\[
\mathcal{C}(W_1, W_2)^\circ=(W_1^\circ \cap (-W_2)^\circ)^\circ
=\left(W_1\cup (-W_2)\right)^{\circ\circ}=\mbox{ \rm s-conv}\left(W_1\cup (-W_2)\right).
\]
\par
(6) It follows by the assertion (5) and Lemma \ref{lemmas-conv}.
\end{proof}
Notice that if we take $W_2=-W_1$, by the assertion (4) of Proposition \ref{propo1}, it follows that
\[
\mathcal{C}(W_1, -W_1)= W_1^\circ.
\]
This means the definition of the set $\mathcal{C}(W_1, W_2)$ 
is a generalization of the definition of spherical polarity.
From this observation, we have the following:
\begin{proposition}\label{theoremisometry}
The mapping 
\[
\mathcal{F}:\overline{\mathcal{H}(S^d, P)}\times \overline{\mathcal{H}(S^d, -P)}\to \overline{\mathcal{H}(S^d, P)}
\]
defined by
\[
\mathcal{F}(W, Y)=\mathcal{C}(W, Y)
\]
always satisfies the following inequality,
\[
h(\mathcal{F}(W_1, Y_1), \mathcal{F}(W_2, Y_2))
\leq 
\max\{h(W_1, W_2), h(Y_1, Y_2)\}. 
\]
Here, $h(W, Y)$ means the Pompeiu-Hausdorff metric of two non-empty compact set $W, Y$ of $S^d$.
\end{proposition}
By the definition of $\mathcal{F}$, Theorem \ref{isometry} can be stated as following:
The mapping 
\[
\overline{\mathcal{F}}:\overline{\mathcal{H}(S^d, P)}\to \overline{\mathcal{H}(S^d, P)}
\]
defined by
\[
\overline{\mathcal{F}}(W)=\mathcal{F}(W, -W)= W^\circ,
\]
is an isometry with respect to Pompeiu-Hausdorff metric.
\begin{example}
\begin{enumerate}
\item If $Y_1=-W_1, Y_2=-W_2$, then we have the equality
\begin{align*}
h(\mathcal{F}(W_1, Y_1), \mathcal{F}(W_2, Y_2))&=h(\mathcal{F}(W_1, -W_1), \mathcal{F}(W_2, -W_2))
\\
&=h(W_1^\circ, W_2^\circ)=h(W_1, W_2)
\\
&=
\max\{h(W_1, W_2), h(Y_1, Y_2)\}. 
\end{align*}
\item If $W_1\neq W_2, Y_1=Y_2$ and $W_1^\circ\cap W_2^\circ=-Y_1^\circ$, then we have the inequality
\begin{align*}
0&=h(-Y_1^\circ, -Y_1^\circ)=h(W_1^\circ\cap-Y_1^\circ, W_2^\circ\cap -Y_1^\circ)=
h(\mathcal{F}(W_1, Y_1), \mathcal{F}(W_2, Y_2))\\
&=\min\{h(W_1, W_2), h(Y_1, Y_1)\}< \max\{h(W_1, W_2), h(Y_1, Y_2)\}.
\end{align*}
\item  If $W_1=W_2, Y_1\neq Y_2, -(Y_1^\circ\cup Y_2^\circ)\subset W_1^\circ$ and 
$Y_1^\circ\cap Y_2^\circ$ with non-empty interior, then 
it follows that
\begin{align*}
0 &=\min\{h(W_1, W_2), h(Y_1, Y_2)\}\\
&<\max\{h(W_1, W_2), h(Y_1, Y_2)\}
=h(Y_1, Y_2) \\
&=h(W_1^\circ\cap -Y_1^\circ, W_2^\circ\cap -Y_2^\circ)
=
h(\mathcal{F}(W_1, Y_1), \mathcal{F}(W_2, Y_2)).
\end{align*}
\end{enumerate}
\end{example}
The following Lemmas are needed for the proof of Proposition \ref{theoremisometry}. 
\begin{lemma}[\cite{hannishimura}]\label{inteq}
For any $W \subset S^d$ such that $W^\circ$ is a spherical convex set and any $0<r<\pi/2$,
\[
\overline{
B\left(
\bigcap_{P\in W}H(P, r)
\right)
}
=
\bigcap_{P\in W}
\overline{B(H(P), r)}.
\]
\end{lemma}
\begin{lemma}\label{inteq1}
Let $W$ (resp.,$Y$)be a subset of $\overline{\mathcal{H}(S^d, P)}$ (resp., $\overline{\mathcal{H}(S^d, -P)}$).  
Let $0<a<\pi/2$ be a positive number. 
Then the following equality holds,
\[
\overline{B(W^\circ, a)}\cap \overline{B(-Y^\circ, a)}=\overline{B(W^\circ\cap -Y^\circ, a)}.
\]
\end{lemma}
\begin{proof}
By Lemma \ref{inteq}, it follows that
 \begin{align*}
\overline{B(W^\circ, a)}\cap \overline{B(-Y^\circ, a)}
&=
\overline{
B\left(
\bigcap_{P\in W}H(P), a
\right)}
\cap
\overline{
B\left(
\bigcap_{Q\in -Y}H(Q), a
\right)}\\
&=
\bigcap_{P\in W}
\overline{B(H(P), a)}
\cap
\bigcap_{Q\in -Y}
\overline{B(H(Q), a)}\\
&=\bigcap_{P\in (W\cup (-Y))}\overline{B(H(P), a)}\\
&=\overline{B(\bigcap_{P\in (W\cup (-Y))}H(P), a)}\\
&=\overline{B((W\cup (-Y))^\circ, a)}.
 \end{align*}
Then by Lemma \ref{lemma1} and Lemma \ref{lemma2.1}, we know that
\[
\overline{B(W^\circ, a)}\cap \overline{B(-Y^\circ, a)}=\overline{B(W^\circ\cap -Y^\circ, a)}.
\]
\end{proof}
Now we are in the position to prove Proposition \ref{theoremisometry}.
\begin{proof} 
Set 
$
\max\{h(W_1, W_2), h(Y_1, Y_2)\}=a.
$
By Theorem \ref{isometry}, it follows that
\[
h(W_1, W_2)=h(W_1^\circ, W_2^\circ)\leq a \ \mbox{and}\ h(Y_1, Y_2)= h(Y_1^\circ, Y_2^\circ)\leq a.
\]
Let $P$ be a point of $\mathcal{C}(W_1, Y_1)$. Since $\mathcal{C}(W_1, Y_1)=W_1^{\circ}\cap (-Y_1^\circ)$, we know that
\[
P\in W_1^{\circ}\cap (-Y_1^\circ).
\]
Then by Lemma \ref{inteq1}, it follows that 
\[
P\in \overline{B(W_2^\circ, a)}\cap \overline{B(-Y_2^\circ, a)}=\overline{B(W_2^\circ\cap -Y_2^\circ, a)}.
\]
Therefore, we have the following inclusion
\[
\tag{$1$}
\mathcal{C}(W_1, Y_1)=
\left(W_1^{\circ}\cap (-Y_1^\circ)\right)
\subset 
\left(
\overline{B(W_2^\circ\cap -Y_2^\circ, a)}
\right)
=\overline{B(\mathcal{C}(W_2, Y_2), a)}.
\] 
In similar, we have the inclusion 
\[
\tag{$2$}
\mathcal{C}(W_2, Y_2)=
\left(W_2^{\circ}\cap (-Y_2^\circ)\right)
\subset 
\left(
\overline{B(W_1^\circ\cap -Y_1^\circ, a)}
\right)
=\overline{B(\mathcal{C}(W_1, Y_1), a)}.
\] 
Thus, by (1) and (2), we have the inequality
\[
h(\mathcal{C}(W_1, Y_1), \mathcal{C}(W_2, Y_2))
\leq 
a=
\max\{h(W_1, W_2), h(Y_1, Y_2)\}.
\]
\end{proof}
A spherical convex body $W$ is said to be {\it self-dual} if $W=W^\circ$.   
By the assertion (4) of Proposition \ref{propo1}, 
a spherical convex body $W$ is self-dual if and only if  
$W=\mathcal{C}(W, -W)$. 
It is known that $W$ is self-dual if and only if $W$ is of constant width $\pi/2$ if and only if 
$W$ is of constant diameter $\pi/2$.
For more details on width, thickness and diameter of spherical convex bodies see for instance \cite{lassak2015}--\cite{lassak20}, \cite{michal21}.
\begin{proposition}\label{propo3}
Let $\mathcal{X}=\mathcal{C}(W_i, Y_i), i\in \Lambda$. 
Then $\widehat{\mathcal{X}}$ 
is the separation center set of 
$\cup_{i\in \Lambda} W_i$ and $\cup_{i\in \Lambda} Y_i$.
\end{proposition}
\begin{proof}
By the definition, it is sufficient to prove that 
\[
\mathcal{X}=\left(\cup_{i\in \Lambda} W_i\right)^\circ\cap\left(\cup_{i\in \Lambda} Y_i\right)^\circ.
\]
Let $P$ be a point of $\mathcal{X}$. Then, by the assumption it follows that 
$P\cdot Q\geq 0$ for any $Q\in W_i$ 
and $P\cdot R\leq 0$ for any $R\in Y_i$, 
where $i\in \Lambda$.
Therefore, $\mathcal{X}$ is a subset of 
$\left(\cup_{i\in \Lambda} W_i\right)^\circ\cap\left(\cup_{i\in \Lambda} Y_i\right)^\circ$. 
Conversely, let $P$ be a point of $\left(\cup_{i\in \Lambda} W_i\right)^\circ\cap\left(\cup_{i\in \Lambda} Y_i\right)^\circ.$ 
Then it is clear that $P$ is a separation center of $\cup_{i\in \Lambda} W_i$ and 
$\cup_{i\in \Lambda} Y_i$. 
This completes the proof.
\end{proof}
\begin{corollary}\label{coro1}
Let $\mathcal{X}=W_i\cap Y_i$, where $W_i, Y_i$ are spherical convex sets and $i=1, \dots, m$. 
We have the following:
\begin{enumerate}
\item
$\widehat{\mathcal{X}}$ is the separation center set of 
$\cup_{i\in \Lambda} W_i^\circ$ and $\cup_{i\in \Lambda} (-Y_i^\circ)$.
\item $\widehat{\mathcal{X}}$ is the separation center set of 
\[
\mbox{\rm s-conv} (\cup_{i\in \Lambda} W_i^\circ)\ \mbox{and}\ \mbox{\rm s-conv} (\cup_{i\in \Lambda} (-Y_i^\circ)).
\]
\end{enumerate}
\end{corollary}
\begin{proof}
(1) By Proposition \ref{recog}, $\widehat{\mathcal{X}}$ is the separation center set of $W_i^\circ$ and 
$-Y^\circ, i\in \Lambda$. Namely,
$\mathcal{X}=\mathcal{C}(W_i^\circ, -Y_i^\circ)$ 
for any $i\in \Lambda$.
By Proposition \ref{propo3}, $\widehat{\mathcal{X}}$ is the separation center set of 
$\cup_{i\in \Lambda} W_i^\circ$ 
and $\cup_{i\in \Lambda} (-Y_i^\circ)$.
\par
(2) Let $P$ be a point of $\mathcal{X}$.
By the assertion (1) of Corollary \ref{coro1}, we know that
 $P\cdot Q_i\geq 0$ for any $Q_i\in W_i^\circ$ and $P\cdot R_i\leq 0$ for any 
   $R_i \in (-Y_i^\circ)$,
where $i\in \Lambda$. 
This implies if $P$ is a point of $\mathcal{X}$
then $P\cdot (tQ_i+(1-t)Q_j)\geq 0$ and $P\cdot (tR_i+(1-t)R_j)\leq 0$,
for any $Q_i\in W_i^\circ, Q_j\in W_j^\circ $ and $R_i \in (-Y_i^\circ), R_j \in (-Y_j^\circ)$.
Here $i, j=1, \dots, m$ and $t\in [0,1]$. 
Therefore, $\mathcal{X}$ is a subset of the separation center set of 
$\cup_{i\in \Lambda} W_i^\circ$ 
and $\cup_{i\in \Lambda} (-Y_i^\circ)$.
Conversely, let $P$ be a separation center of 
$\mbox{\rm s-conv} (\cup_{i\in \Lambda} W_i^\circ)$ and $\mbox{\rm s-conv} (\cup_{i\in \Lambda} (-Y_i^\circ))$.
Without loss of generality, we may assume that 
$P\cdot Q_i\geq 0$ for any $Q_i\in W_i^\circ$ and $P\cdot R_i\leq 0$ 
for any $R_i \in (-Y_i^\circ)$, where $i\in \Lambda$. 
This means 
$P\in (W_i^\circ)^\circ$ and $-P \in (-Y_i^\circ)^\circ$, where $i\in \Lambda$. 
Then, by Lemma \ref{lemmas-conv}, one obtains that  
$P\in W_i$ and $P\in Y_i$ for any  $i\in \Lambda$. 
Therefore, the separation center set of 
$\mbox{s-conv} (\cup_{i=1}^m W_i^\circ)$ and $\mbox{s-conv} (\cup_{i=1}^m (-Y_i^\circ))$
is a subset of $\widehat{\mathcal{X}}$.
\end{proof}
\begin{corollary}
Let $\mathcal{C}(W_1, W_2)$ be a non-empty set, where $W_1$ and $W_2$ are spherical closed sets. Then it follows that
\[
\mathcal{C}(W_1, W_2)=(\mbox{\rm s-conv}(W_1))^{\circ}\cap (\mbox{\rm s-conv}(-W_2))^{\circ},
\]
\end{corollary}
\begin{proof}
By the proof of the assertion (2) of Corollary \ref{coro1}, we know that if $P$ is a separation center of $W_1$ and $W_2$, 
then it is a separation center of $\mbox{s-conv}(W_1)$ and $\mbox{s-conv}(W_2)$. 
Then, by Proposition \ref{recog} and $W^\circ=\mbox{s-conv}(W)^\circ$, 
one obtains that
\[
\mathcal{C}(W_1, W_2)=(\mbox{s-conv}(W_1))^{\circ}\cap (\mbox{s-conv}(-W_2))^{\circ}.
\]
Conversely, if $P$ is a separation center of $\mbox{s-conv}(W_1)$ and $\mbox{s-conv}(-W_2)$, 
then it is clear that $P$ is a separation center of of $W_1$ an $-W_2$.
\end{proof}
\begin{corollary}\label{corolargest}
Let $\mathcal{X}$ be a closed spherical convex set. Then $\mathcal{X}^\circ, -\mathcal{X}^\circ$ is the largest pair separated by $\widehat{\mathcal{X}}$.
\end{corollary}
\begin{proof}
Since $\mathcal{X}=\mathcal{X}\cap \mathcal{X}$, by Proposition \ref{recog}, 
the equality $\mathcal{X}=\mathcal{C}(\mathcal{X}^\circ, -\mathcal{X}^\circ)$ holds.
This means $\widehat{\mathcal{X}}$ is the separation center set of 
$\mathcal{X}^\circ, -\mathcal{X}^\circ$. 
Moreover, the assersion (2) of Proposition \ref{propo1} implies  
$W_1$ (resp. $W_2$) is a subset of $\mathcal{X}^\circ$ 
(resp. $-\mathcal{X}^\circ$) 
for any $W_1, W_2$ 
satisfies $\mathcal{X}=\mathcal{C}(W_1, W_2)$. 
This completes the proof.
\end{proof}
\begin{corollary}\label{coro4}
Let $W_1, W_2$ 
be nonempty disjoint spherical closed convex sets. 
Then the following equality holds.
\[
\mathcal{C}(W_1, W_2)
=\mathcal{C}\left(
  \mbox{ \rm s-conv}
      \left(
          W_1\cup (-W_2)
            \right),
            -\mbox{ \rm s-conv}\left(W_1\cup (-W_2)
      \right) \right).
      \]
\end{corollary}
\begin{proof}
By the proof of Corollary \ref{corolargest} and the assertion (5) of Proposition \ref{propo1},
it follows that
\begin{align*}
\mathcal{C}(W_1, W_2)
&=\mathcal{C}\left(
\mathcal{C}(W_1, W_2)^\circ, -\mathcal{C}(W_1, W_2)^\circ
\right)\\
&=\mathcal{C}\left(
  \mbox{ \rm s-conv}
      \left(
          W_1\cup (-W_2)
            \right),
            -\mbox{ \rm s-conv}\left(W_1\cup (-W_2)
      \right) \right).
\end{align*}
\end{proof}
\subsection{Case 2: $W_i=\{P_i\}, i=1, \dots, m$.}
\indent
\par
By the definition of separation center set, the following Lemma is obvious.
\begin{lemma}\label{prop5}
Let $\{P_i\}_{i=1}^m$ be a finite points of $S^d$. 
Let 
\[
\mathcal{B}_0=\{Q\in S^d\mid (\cup_{i=1}^m P_i) \cap H(Q)=\emptyset \ \mbox{or}\ (\cup_{i=1}^m P_i)\cap H(-Q)=\emptyset \}.
\] 
Then $S^d \backslash\mathcal{B}_0$ is the separation center set of $\{P_i\}_{i=1}^m$.
\end{lemma}
\begin{proposition}\label{prop6}
Let $\{P_i\}_{i=1}^m$ be a finite points of $S^d$ and let 
\begin{align*}
\mathcal{B}_j=\{Q\in S^d\backslash \mathcal{B}\mid\ 
&\mbox{there exist exactly}\ P_{i_1}, \dots, P_{i_j} of\{P_i\}_{i=1}^m\ \\
&\mbox{such that}\ Q\cdot P_{i_1}
=\dots= Q\cdot P_{i_j}=0\},
\end{align*}
where $1\leq j \leq m$ and $\{i_1, \dots, i_j\}\subset \{1, \dots, m\}$. 
Then the separation center set of $\{P_i\}_{i=1}^m$ is dense in 
$S^d\backslash \left(\mathcal{B}_0\bigcup (\cup_{j=k}^l\mathcal{B}_j)\right)$ for $1\leq k\leq l\leq m$.
\end{proposition}
\begin{proof}
By Lemma \ref{prop5}, it is sufficient to prove that 
$\cup_{j=k}^l\mathcal{B}_j$ ($1\leq k\leq l\leq m$) is an empty interior set in $S^d$.
Since
 $\mathcal{B}_j$ is the finite union of the intersection of $S^d$ and the hyperplane determined by $Q\cdot P_{i_k}=0$  (where $1\leq k \leq j$),
it follows that $\mathcal{B}_j\subset S^d$ is a spherical closed set with empty interior. 
Therefore, the union $\cup_{j=k}^l\mathcal{B}_j$ is a spherical closed set with empty interior. 
Thus, we have the conclusion. 
\end{proof}
\begin{proposition}\label{coro5}
Let $\{P_i\}_{i=1}^m$ be a finite points of $S^d$. Then the union 
$\cup_{i=1}^m P_i$ is not hemispherical if and only if the separation center set of $\{P_i\}_{i=1}^m$ is dense in $S^d$. 
\end{proposition}
\begin{proof}
Suppose that the union $\cup_{i=1}^m P_i$ is not hemispherical. 
Then for any point $P\in S^d$, the intersection $(\cup_{i=1}^m P_i)\cap H(P)$ is a non-empty set. 
This means that the set $\mathcal{B}_0$ is empty. 
By Lemma \ref{prop5}, the separation center set $S^d\backslash \mathcal{B}_0$ is dense in $S^d$.
\par
 Conversely, we assume that the separation center set of $\{P_i\}_{i=1}^m$ is dense in $S^d$. 
Suppose that the union $\cup_{i=1}^m P_i$ is a hemispherical set.
Then there exists a point $Q\in S^d$ such that
\[
\left(\bigcup_{i=1}^m P_i\right)\bigcap H(Q)=\emptyset.
\]
This implies 
\[
\mbox{s-conv}\left(\bigcup_{i=1}^m P_i\right)\bigcap H(Q)=\emptyset.
\]
Therefore, it follows that there exists a sufficiently small $\varepsilon>0$ such that
\[
\mbox{s-conv}\left(\bigcup_{i=1}^m P_i\right)\bigcap B(H(Q), \varepsilon)=\emptyset.
\]
By Lemma \ref{lemmaobvi}, we have that
\[
\mbox{s-conv}\left(\bigcup_{i=1}^m P_i\right)\bigcap \bigcup_{R\in B(Q, \varepsilon)}H(R)=\emptyset.
\]
Thus, any point $R\in B(Q, \varepsilon)$ is not a separation center point of 
$\cup_{i=1}^m P_i$. 
This contradicts the assumption that 
the separation center set of $\{P_i\}_{i=1}^m$ is dense in $S^d$. 
\end{proof}
\par
 \section{A Spherical polytope as a semi-separation center set of spherical polytopes}
\subsection{Spherical partition pairs}Since  $W$ is a spherical polytope if and only if its polar body $W^\circ$ is a spherical polytope, 
by Proposition \ref{propo1}, it follows that $\mathcal{C}(W_1, W_2)$ is a spherical polytope
if and only if $\mbox{s-conv}(W_1\cup W_2)$ is a spherical polytope. 
\indent
\par
For any given spherical polytope $\mathcal{P}$ with vertices  
$P_1, \dots, P_m$, by Proposition \ref{recog}, the union $\widehat{\mathcal{P}}=\mathcal{P}\cup -\mathcal{P}$ 
is a separation center set of some suitable sets. 
A pair of spherical polytopes $(\mathcal{P}_1, \mathcal{P}_2)$ is said to be 
{\it a spherical partition pair of} $\mathcal{P}$ if 
satisfies the condition ($\star$) with respect to $\mathcal{P}$ below:
\begin{enumerate}\label{conditionstar}
\item $\mathcal{P}_1$ is a spherical polytope with vertices $P_i, i\in I$ and 
$\mathcal{P}_2$ is a spherical polytope with vertices $P_j, j\in J$, where, $(I, J)$ is a partition of 
$\{1,\dots, m\}$.
\item $\mathcal{P}_1\cap \mathcal{P}_2=\emptyset.$
\end{enumerate}
for  given polytope $\mathcal{P}$ generated by 
$P_1, \dots, P_m$.
\bigskip
\par
By Corollary \ref{coro4} and Corollary \ref{corolargest}, we have the following:
\begin{theorem}\label{polytopesepa}
Let $(\mathcal{P}_1, -\mathcal{P}_2)$ and $(\mathcal{P}_3, -\mathcal{P}_4)$ 
be two pairs of spherical polytopes satisfy the condition ($\star$) with respect to a spherical polytope $\mathcal{P}$.
Then pairs $(\mathcal{P}_1, -\mathcal{P}_2)$ and $(\mathcal{P}_3, -\mathcal{P}_4)$ have same separating center set.
\end{theorem}
\begin{proof}
Since the pairs $(\mathcal{P}_1, -\mathcal{P}_2)$ and $(\mathcal{P}_3, -\mathcal{P}_4)$ satisfy the (1) of condition ($\star$), it follows that $\mathcal{P}_1\cap-\mathcal{P}_2=\mathcal{P}_3\cap -\mathcal{P}_4=\emptyset.$
By Corollary \ref{coro4}, it follows that
\[
\mathcal{C}(\mathcal{P}_1, -\mathcal{P}_2)=\mathcal{C}\left(\mbox{ \rm s-conv}
      \left(
         \mathcal{P}_1\cup \mathcal{P}_2
            \right),
            -\mbox{ \rm s-conv}\left(\mathcal{P}_1\cup \mathcal{P}_2
      \right)\right) 
                 \]
 and 
 \[
\mathcal{C}(\mathcal{P}_3, -\mathcal{P}_4)=\mathcal{C}\left(\mbox{ \rm s-conv}
      \left(
         \mathcal{P}_3\cup \mathcal{P}_4
            \right),
            -\mbox{ \rm s-conv}\left(\mathcal{P}_3\cup \mathcal{P}_4
      \right)\right). 
            \]           
By the assumption, the pairs $(\mathcal{P}_1, -\mathcal{P}_2)$ and $(\mathcal{P}_3, -\mathcal{P}_4)$ satisfy the (2) of condition ($\star$), this means
\[
 \mathcal{C}(\mathcal{P}_1, -\mathcal{P}_2)=
 \mathcal{C}(\mathcal{P}_3, -\mathcal{P}_4)
 =\mathcal{C}\left(
          \mathcal{P},
            -\mathcal{P} \right)=
          \mathcal{P}^\circ,
\]
where 
$\mathcal{P}=\mbox{ \rm s-conv}\left(\mathcal{P}_1\cup \mathcal{P}_2\right)
=\mbox{ \rm s-conv}\left(\mathcal{P}_3\cup \mathcal{P}_4\right)$.
Therefore, the separation center sets of pairs 
$(\mathcal{P}_1, -\mathcal{P}_2)$ and $(\mathcal{P}_3, -\mathcal{P}_4)$ are same, namely, 
\[
\tag{$\star\star$} 
\mathcal{C}(\mathcal{P}_1, -\mathcal{P}_2)\cup -\mathcal{C}(\mathcal{P}_1, -\mathcal{P}_2)
=
\mathcal{C}(\mathcal{P}_3, -\mathcal{P}_4)\cup -\mathcal{C}(\mathcal{P}_3, -\mathcal{P}_4)
=
\mathcal{P}^\circ\cup -\mathcal{P}^\circ.
\]
\end{proof}
\subsection{Convex set(body) mappings}
Let $\mathcal{U}, \mathcal{V}$ be regular submanifolds of $\mathbb{R}^d, \mathbb{R}^l$ respectively. 
A subset $U\subset \mathcal{U}$ is said to be {\it convex} in $\mathcal{U}$ if $g_{{}_{PQ}}$ is contained in $U$ for any $P, Q\in U$, 
where $g_{{}_{PQ}}$ is the shortest geodesic curve connecting $P, Q$. 
The smallest convex set containing $U$ is said to be {\it the convex hull of $U$}, 
denoted by $\mbox{\rm conv}(U)$.
A mapping $f: \mathcal{U}\to \mathcal{V}$ is said to be a {\it convex set mapping} (resp. {\it convex body mapping}) if the pre-image 
$f^{-1}(V)$ is a convex set (resp. convex body) of $\mathcal{U}$.
Here
\[
f^{-1}(V)=\{x\in \mathcal{U}\mid f(x)\in V\}
\] 
and  $V$ is a convex set (resp. convex body) of $\mathcal{V}$.
\begin{proposition}
Let $f$ be a mapping from $\mathbb{R}^d$ onto $\mathbb{R}^l$ such that pre-image of any half space of $\mathbb{R}^m$ 
is a convex set of $\mathbb{R}^n$. 
Then $f$ is a convex set mapping.
\end{proposition}
\begin{proof}
Let $V\subset \mathbb{R}^l$ be a closed convex set. 
Since every closed convex set in $\mathbb{R}^d$ is the intersection of its supporting half-spaces (\cite{schneider}, Corolary 1.3.5), 
$V$ can be written as 
\[
V=\bigcap_{\Gamma\in \mathcal{H}}\Gamma,
\]
where $\mathcal{H}$ is the collection of closed half-space that contains $V$. 
Then it follows that
\[
f^{-1}(V)=f^{-1}(\bigcap_{\Gamma\in \mathcal{H}}\Gamma)=\bigcap_{\Gamma\in \mathcal{H}}f^{-1}(\Gamma),
\]
Thus, by the assumption, we know that $f^{-1}(V)$ is a closed convex set.
\end{proof}
As a corollary we have the following:
\begin{corollary}
Let $f$ be a mapping from $\mathbb{R}^d$ onto $\mathbb{R}^l$ such that pre-image of any half space of $\mathbb{R}^l$ 
is a half-sapce of $\mathbb{R}^d$. 
Then $f$ is a convex set mapping.
\end{corollary}
\begin{example} Some examples of convex set (convex body) mapping.
\begin{enumerate}
\item Any affine mapping $f:\mathbb{R}^d\to\mathbb{R}^l$ is a convex set mapping (convex body mapping).
\item The central projection 
$\alpha_{{}_N}$ 
is a convex set (convex body) mapping. 
The inverse mapping of $\alpha_N$ is also a convex set (convex body) mapping.
\item The stereo projection $\phi: S^d\backslash N \to \mathbb{R}^d\times \{-1\}$, defined by
\[
\phi\left(x_1, \ldots, x_{n}, x_{n+1}\right)
=
(-\frac{2x_1}{x_{n+1}}, -\frac{2x_2}{x_{n+1}}, \dots, -\frac{2x_n}{x_{n+1}}, -1),
\]
is not a convex set (convex body) mapping. 
This is because of the pre-image of the line $ \{(m, 0, \dots 0, -1)\mid m\in \mathbb{R}\}$ 
of  $\mathbb{R}^d\times \{-1\}$ is not hemispherical. 
\end{enumerate}
\end{example}
\begin{proposition}
Let $\mathcal{U}$ (resp. $\mathcal{V}$) be a regular sub-manifold of $\mathbb{R}^d$ (resp. $\mathbb{R}^l$). 
Let $f: \mathcal{U}\to \mathcal{V}$  and its inverse mapping be convex set mappings (or convex body mappings). 
Then we have the following equality
\[
f(\mbox{\rm conv}(U))=\mbox{\rm conv}(f(U)).
\] 
for any $U\subset \mathcal{U}$.
\end{proposition}
\begin{proof}
Set 
\[
\mbox{ conv}(f(U))=\bigcap V^\prime,
\]
where $V^\prime$ is a convex set (or convex body) of $\mathcal{V}$ containing $f(U)$.
Then it follows that
\[
f^{-1}(\mbox{ conv}(f(U)))=f^{-1}(\cap V^\prime).
\]
Since $f(U)$ is a subset of $\bigcap V^\prime$, by the assumption 
$f^{-1}(\cap V^\prime)$ is convex, we know that 
\[
\mbox{ conv}(U)\subset f^{-1}(\cap V^\prime).
\]
This means that 
\[
f(\mbox{\rm conv}(U))\subset f(f^{-1}(\cap V^\prime))=\bigcap V^\prime=\mbox{ conv}(f(U)).
\]
Conversely, set 
\[
\mbox{ conv}(U)=\bigcap U^\prime,
\]
where $U^\prime$ is a convex set (or convex body) of $\mathcal{U}$ containing $U$.
Since $f$ is an injective convex set mapping, we have that 
\[
f(\mbox{conv}(U))=f(\cap U^\prime)=\cap f(U^\prime).
\]
Notice that $f(U)$ is a subset of $f(U^\prime)$ and $\cap f(U^\prime)$ is convex.
Thus, we can conclude that
\[
\mbox{\rm conv}(f(U))\subset f(\mbox{ conv}(U)).
\]
\end{proof}
\begin{proposition}
Let $\mathcal{U}$ (resp. $\mathcal{V}$) be a regular submanifold of $\mathbb{R}^d$ (resp. $\mathbb{R}^l$). If the equality 
\[
f(\mbox{\rm conv}(U))=\mbox{\rm conv}(f(U))
\] 
holds for any $U\subset \mathcal{U}$.
Then $f: \mathcal{U}\to \mathcal{V}$ is a convex set mapping.
\end{proposition}
\begin{proof}
Let $V$ be a convex set in $\mathcal{V}$. 
It is clear that $\mbox{\rm conv}(V)=V$. 
Set $U=f^{-1}(V)$. Since 
\[
f(\mbox{\rm conv}(U))=\mbox{\rm conv}(f(U))=V,
\]
it follows that 
\[
U=f^{-1}(V)=f^{-1}(f(\mbox{\rm conv}(U)))=\mbox{\rm conv}(U).
\]
This implies $U$ is a convex set.
Thus, the mapping $f$ is a convex set mapping.
\end{proof}
\subsection{Number of the spherical partition pair}Let $\mathcal{U}$ be a smooth subspace of $\mathbb{R}^d$ and 
let $\{U_1, \dots, U_p\}$ be a finite closed sets of $\mathcal{U}$. 
A pair $(\mathcal{U}_1, \mathcal{U}_2)$ is said to be 
{\it a separation of} $\{U_1, \dots, U_p\}$ if it satisfies the condition ($\star\star$) below
\begin{itemize}
\item there exist nonempty subsets  $I, J$ of $\{1,\dots, p\}$ such that 
\[
I\cap J =\emptyset, I\cup J=\{1,\dots, p\}
\]
 and 
\[
\mathcal{U}_1=\mbox{conv}(\bigcup_{i\in I}U_i),\ \mathcal{U}_2=\mbox{conv}(\bigcup_{j\in J}U_j).
\] 
\item $\mathcal{U}_1\cap \mathcal{U}_2=\emptyset.$
\end{itemize}
Denote by $\#\{U_1, \dots, U_p\}$ the sum of the number of pairs $(\mathcal{U}_1, \mathcal{U}_2)$ 
satisfies the condition ($\star\star$).
\begin{lemma}\label{nuberinequality}
Let $\mathcal{U}$ (resp. $\mathcal{V}$) be a subspace of $\mathbb{R}^d$ (resp. $\mathbb{R}^l$) and 
let $\{U_1, \dots, U_p\}$ (resp. $\{V_1, \dots, V_p\}$) be a finite closed sets of $\mathcal{U}$ (resp. $\mathcal{V}$). 
Let $f: \mathcal{U}\to \mathcal{V}$  be a  convex set mapping (or convex body mapping) such that 
\[
f(U_i)=V_i,\ \mbox{for}\ i\in\{1, \dots, p\}.
\]
Then the following inequality holds:
\[
\#\{V_1, \dots, V_p\}
\leq
\#\{U_1, \dots, U_p\}.
\]
Especially, if $f$ and its inverse $f^{-1}$ are convex set (convex body) mapping, then
\[
\#\{V_1, \dots, V_p\}
=
\#\{U_1, \dots, U_p\}.
\]
\end{lemma}
\begin{proof}
Let $(\mathcal{V}_1, \mathcal{V}_2)$ be a separation pair of $\{V_1, \dots, V_p\}$. 
Set 
\[
\mathcal{U}_1=f^{-1}(\mathcal{V}_1), \mathcal{U}_2=f^{-1}(\mathcal{V}_2).
\]
Since $f$ is a convex set (convex body) mapping, it follows that 
$U_i\in \mathcal{U}_1, U_j\in \mathcal{U}_2$ for $i\in I, j\in J$
and
$\mathcal{U}_1, \mathcal{U}_2$ are disjoint convex sets in $\mathcal{U}$. 
This implies that 
\[
\mbox{conv}(\bigcup_{i\in I}U_i)\bigcap\mbox{conv}(\bigcup_{j\in J}U_j)=\emptyset.
\]
Therefore, it follows that 
\[
\#\{V_1, \dots, V_p\}
\leq
\#\{U_1, \dots, U_p\}.
\]
\end{proof}
\section{Proof of Theorem \ref{theoremulb}}
Notice that for any spherical polytope $\mathcal{P}$ in $S^d$, $\mathcal{P}=\mathcal{P}^{\circ\circ}$ holds (Lemma \ref{lemmas-conv}).
By Lemma \ref{lemmasconv} and Lemma \ref{lemma1}, the following equality holds:
\[
\mathcal{P}=\mbox{s-conv}(\mathcal{P}_1\cup -\mathcal{P}_2)=
(\mathcal{P}_1\cup -\mathcal{P}_2)^{\circ\circ}=
(\mathcal{P}_1^\circ\cap \mathcal{P}_2^\circ)^\circ
=\mathcal{C}(\mathcal{P}_1,-\mathcal{P}_2)^\circ.
\]
Since $\mathcal{P}^\circ=\mathcal{C}(\mathcal{P}_1,-\mathcal{P}_2)$,
we know that the number $\#_f\mathcal{P}^\circ$ is determined by the sum of the number of  spherical partition pair $(\mathcal{P}_1,-\mathcal{P}_2)$ of $\mathcal{P}$. 
Namely, 
the number $\#_f\mathcal{P}^\circ$ is determined by
the number of pairs $(\mathcal{P}_1,-\mathcal{P}_2)$
satisfy the condition ($\star$) with respect to $\mathcal{P}$ 
and $\mathcal{P}_1$ (or $-\mathcal{P}_2$) is a face of $\mathcal{P}$.
Since the central projection $\alpha_N$ 
and its inverse mappings are convex set mappings, 
by Lemma \ref{nuberinequality}, it follows that $\#_f\mathcal{P}^\circ/2$ equals 
the sum of the
face number of $\alpha_N(\mathcal{P}^\circ)$. 
Since $\mathcal{P}$ is a simple spherical polytope with $p$ facets, by duality,
$\mathcal{P}^\circ$ is a simplicial spherical polytope with $p$ vertices. 
Then, by the Upper Bound Theorem (Theorem \ref{UBT}) and the Lower Bound Theorem (Theorem \ref{LBT}) we have the following inequality:
\[
p+\sum_{j=1}^{d-1}\varphi_{d-1-j}(d,p)\leq\ \mbox{the sum of the face numbers of}\ \alpha_N(\mathcal{P})\leq \sum_{j=0}^{d-1}\Phi_{d-1-j}(d,p).
\]
Notice that  $(\mathcal{P}_1,-\mathcal{P}_2)$ is a spherical partition pair of $\mathcal{P}$ 
if and only if  $(\mathcal{P}_2, -\mathcal{P}_1)$ is a spherical partition pair of $\mathcal{P}$ 
(by the assertion (1) of Proposition \ref{propo1}).
Thus, we conclude that
\[
2\left(p+\sum_{j=1}^{d-1}\varphi_{d-1-j}(d,p)\right)\leq\#_f\mathcal{P}\leq 
2 \left(p+\sum_{j=0}^{d-1}\Phi_{d-1-j}(d,p)\right).
\]
\bigskip
\begin{question}
{\rm 
Let $\mathcal{P}$ be a simple spherical polytope with $p$ facets 
and let $\mathcal{P}_1, \mathcal{P}_2$ be two disjoint spherical polytopes 
such that 
\[
\mathcal{P}=\mbox{{\rm s-conv}}(\mathcal{P}_1\cup -\mathcal{P}_2).
\] 
Then, by the proof of Theorem \ref{polytopesepa} we know that $\mathcal{P}$ is a semi-separation center set of 
$\mathcal(\mathcal{P}_1, \mathcal{P}_2)$.
Let $m$ be a non-negative integer. Suppose that the sum of the common vertices number of 
$\mathcal{P}_1, -\mathcal{P}_2$ is $m$. 
The {\it number of spherical face-partition pair } $(\mathcal{P}_1, \mathcal{P}_2)$ of $\mathcal{P}$
is denoted by $\#_f^m{\mathcal P}$.
Then, what are the upper and lower boundaries of 
$\#_f^m{\mathcal P}$? 
Theorem \ref{theoremulb} asserts that 
\[
2\left(p+\sum_{j=1}^{d-1}\varphi_{d-1-j}(d,p)\right)\leq\#_f^0\mathcal{P}\leq 
2 \left(p+\sum_{j=1}^{d-1}\Phi_{d-1-j}(d,p)\right).
\]
}
\end{question}
\section*{Acknowledgements}
The author would like to thank Professor Takashi Nishimura for his helpful
discussions and comments.
This work was supported by the Research Institute for
Mathematical Sciences, a Joint Usage/Research Center located in Kyoto University.
This work was supported, in partial, by
Natural Science Basic Research Plan in Shaanxi Province of China
(2023-JC-YB-070).  

\end{document}